\documentclass[12pt]{article}
\usepackage{amsmath,amsthm,amssymb,
mathabx,bbm}

\usepackage[all]{xy}
\usepackage[active]{srcltx}

\newcommand{\C}{\mathbb C}
\newcommand{\R}{\mathbb R}
\newcommand{\F}{\mathbb F}

\newcommand{\pair}[4]%
{\xymatrix@C=25pt{{#1}\ar@<0.3ex>[r]^{#3}
\ar@<-0.3ex>[r]_{#4}&{#2}}}

\DeclareMathOperator{\im}{Im}
\DeclareMathOperator{\Ker}{Ker}
\DeclareMathOperator{\alg}{trn}

\renewcommand{\le}{\leqslant}
\renewcommand{\ge}{\geqslant}

\newtheorem{theorem}{Theorem}
\newtheorem{lemma}{Lemma}
\newtheorem{rs}{Transformation}
\newtheorem{Rs}{Transformation}

\theoremstyle{definition}

\begin{document}
\title{Regularizing decompositions for matrix
pencils and a topological
classification of pairs of linear
mappings\thanks{This manuscript was
published in [Linear Algebra Appl. 450
(2014) 121--137], in which Theorem 2
was formulated inaccurately; see the
footnote to Theorem 2.}}

\author{Vyacheslav Futorny\\
University of S\~ao Paulo,
Brazil\\futorny@ime.usp.br
      \and
Tetiana Rybalkina\\
Institute of Mathematics,
Kiev, Ukraine\\
rybalkina\_t@ukr.net
   \and
Vladimir V. Sergeichuk
\\
Institute of Mathematics, Kiev,
Ukraine\\ sergeich@imath.kiev.ua}

\date{}
\maketitle

\begin{abstract}
By Kronecker's theorem, each matrix
pencil $A+\lambda B$ over a field $\F$
is strictly equivalent to its
\emph{regularizing decomposition};
i.e., a direct sum
\[
(I_r+\lambda D)\oplus (M_1+\lambda N_1)
\oplus\dots\oplus (M_t+\lambda N_t),
\]
where $D$ is an  $r\times r$
nonsingular matrix and each
$M_i+\lambda N_i$ is of the form
$I_k+\lambda J_k(0)$, $J_k(0)+\lambda
I_k$, $L_k+\lambda R_k$, or
$L_k^T+\lambda R_k^T$, in which $L_k$
and $R_k$ are obtained from $I_k$ by
deleting its last or, respectively,
first row and $J_k(0)$ is a singular
Jordan block.

We give a method for constructing a
regularizing decomposition of an
$m\times n$ matrix pencil $A+\lambda
B$, which is formulated in terms of the
linear mappings $A,B:\F^n\to\F^m$.

Two $m\times n$ pencils $A+\lambda B$
and $A'+\lambda B'$ over $\mathbb
F=\mathbb R\text{ or }\mathbb C$ are
said to be \emph{topologically
equivalent} if the pairs of linear
mappings $A,B:\F^n\to\F^m$ and
$A',B':\F^n\to\F^m$ coincide up to
homeomorphisms of the spaces $\mathbb
F^{n}$ and $\mathbb F^{m}$.  We prove
that two pencils are topologically
equivalent if and only if their
regularizing decompositions coincide up
to permutation of summands and
replacement of $D$ by a nonsingular
matrix $D'$ such that the linear
operators $D,D':\F^r\to\mathbb F^{r}$
coincide up to a homeomorphism of
$\mathbb F^{r}$.
\medskip

\medskip

\end{abstract}

\section{Introduction}


In this article
\begin{itemize}
  \item a regularizing
      decomposition of a matrix
      pencil is constructed by a
      method that is formulated in
      terms of images, preimages,
      and kernels of linear
      mappings, and

  \item the problem of topological
      classification of pairs of
      linear mappings
$A,B:\F^n\to\F^m$ over $\mathbb
F=\mathbb R\text{ or }\mathbb C$ is
reduced to the
      open problem of topological
      classification of linear
      operators, which was solved
      in special cases (in
      particular, for operators
      without eigenvalues that are
      roots
of~$1$) in \cite{bud1,Capp-conexamp,%
Capp-2th-nas-n<=6,Capp-big-n<6,%
Cap+sha,Cap+ste,h-p,h-p1,
Kuip-Robb, Robb}.
\end{itemize}

\subsection{A regularizing
decomposition of matrix pencils}

A \emph{matrix pencil} over a field
$\F$ is a parameter matrix $A+\lambda
B$, in which $A$ and $B$ are matrices
over $\F$ of the same size. Two matrix
pencils $A+\lambda B$ and $A'+\lambda
B'$ are \emph{strictly equivalent} if
there exist nonsingular matrices $S $
and $R$ over $\F$ such that
$S(A+\lambda B)R=A'+\lambda B'$. This
means that the corresponding matrix
pairs $(A,B)$ and $(A',B')$ are
\emph{equivalent}; i.e.,
\begin{equation}\label{kmr}
\text{$SA=A'R$ and $S B=B'R$ for some
nonsingular $S$ and $R$.}
\end{equation}
In what follows, we consider matrix
pairs $(A,B)$ instead of pencils
$A+\lambda B$.

Denote by $J_k(0)$ the $k\times k$
singular Jordan block with units under
the diagonal. Write
\[
L_k:=\begin{bmatrix}
1\ &\ 0\ &&\ 0\\&\ddots&\ddots&\\0&&\ 1\ &0
\end{bmatrix},\ \
R_k:=\begin{bmatrix}
0\ &\ 1\ &&0\\&\ddots&\ddots&\\0&&\ 0\ &\ 1
\end{bmatrix}\ \
\text{($(k-1)$-by-$k$);}
\]
note that $L_1=R_1=0_{01}$ is the
$0\times 1$ matrix of the linear
mapping $\F\to 0$. Kronecker's
canonical form for matrix pencils (see
\cite[Section XII]{gan}) ensures that
each matrix pair $(A,B)$ over a field
$\F$ is equivalent to a direct sum
\begin{equation}\label{1.4a}
(I_r,D)\oplus(M_1,N_1)\oplus
(M_2,N_2)\oplus\dots\oplus (M_t,N_t)
\end{equation}
in which $D$ is an $r\times r$
nonsingular matrix and each $(M_i,N_i)$
is one of the matrix pairs
\begin{equation}\label{krx}
(I_k,J_k(0)),\ (J_k(0),I_k),\
(L_k,R_k),\ (L_k^T,R_k^T),\quad k=1,2,\dots
\end{equation}
The summands $(M_1,N_1),\dots,
(M_t,N_t)$ are called the
\emph{indecomposable singular summands}
of $(A,B)$; they are determined by
$(A,B)$ uniquely, up to permutation.
The summand $(I_r,D)$ is called a {\it
regular part} of $(A,B)$; its matrix
$D$ is determined uniquely up to
similarity transformations $C^{-1}DC$
($C$ is nonsingular); $D$ can be taken
in the Jordan form if $\F=\C$ or in the
real Jordan form \cite[Section
3.4]{HJ12} if $\F=\R$.

The direct sum \eqref{1.4a} is called a
\emph{regularizing decomposition} of
$(A,B)$. Van Dooren's algorithm
\cite{doo} constructs a regularizing
decomposition for a complex matrix
pencil using only unitary
transformations, which is important for
its numerical stability. The algorithm
was extended to cycles of linear
mappings in \cite{ser_cyc} and to
matrices under congruence and
*congruence in \cite{h-s_lin}.

In Theorem \ref{ykw} we give a method
for constructing a regularizing
decomposition of a matrix pair; the
method is formulated in terms of vector
spaces and linear mappings.

\subsection{A topological
classification of pairs of linear
mappings}

Each matrix $A\in\F^{m\times n}$
defines the linear mapping (which we
denote by the same letter)
$A:\F^n\to\F^m$, $v\mapsto Av$. Let
$\F$ be $\R$ or $\C$. We say that pairs
$(A,B)$ and $(A',B')$ of $m\times n$
matrices are \emph{topologically
equivalent} if the corresponding pairs
of linear mappings
\begin{equation}\label{kyc}
\xymatrix@C=25pt{ {\mathbb F^n}
\ar@<0.3ex>[r]^{A\ } \ar@<-0.3ex>[r]_{B\ }
&{\mathbb F^m}},\qquad
\xymatrix@C=25pt{ {\mathbb F^{n'}}
\ar@<0.3ex>[r]^{A'\; } \ar@<-0.3ex>[r]_{B'\; }
&{\mathbb F^{m'}}}
\end{equation}
are \emph{topologically equivalent},
which means that there exist
homeomorphisms $\varphi:\F^n\to
\F^{n'}$ and $\psi:\F^m\to \F^{m'}$
such that the diagram
\begin{equation}\label{kjlj}
\begin{split}
\xymatrix@R=20pt@C=40pt{
{\mathbb F^n} \ar@<0.3ex>[r]^A
\ar@<-0.3ex>[r]_B
\ar[d]_{\varphi}
&{\mathbb F^m}\ar[d]^{\psi}
   \\
{\mathbb F^{n'}} \ar@<0.3ex>[r]^{A'}
\ar@<-0.3ex>[r]_{B'}&{\mathbb F^{m'}}
}\end{split}
\end{equation}
is commutative:  $\psi A=A'\varphi$ and
$\psi B=B'\varphi.$

A mapping between two topological
spaces is a \emph{homeomorphism} if it
is a continuous bijection whose inverse
is also a continuous bijection; we
consider $\F^n$ as a topological space
with topology induced by the usual
norm: $ \|z\|=\sqrt{z_1\bar
z_1+\dots+z_n \bar z_n} $. If
$\varphi:\F^n\to \F^{n'}$ is a
homeomorphism, then $n=n'$ by
\cite[Corollary 19.10]{Bred} or
\cite[Section 11]{McCl}. This gives
$n=n'$ and $m=m'$ in \eqref{kjlj}.

The pairs of linear mappings
\eqref{kyc} are called \emph{linearly
equivalent} if there exist linear
bijections $\varphi$ and $\psi$ such
that the diagram \eqref{kjlj} is
commutative.  This means that the
corresponding matrix pairs $(A,B)$ and
$(A',B')$ are equivalent; i.e.,
\eqref{kmr} holds. Since each linear
bijection $\F^n\to \F^n$ is a
homeomorphism, if two pairs \eqref{kyc}
are linearly equivalent, then they are
topologically equivalent.

We say that two $r\times r$ matrices
$D$ and $D'$ over $\F$ are
\emph{topologically similar} if there
exists a homeomorphism
$\varphi:\F^r\to\F^r$ such that the
diagram
\[
\xymatrix@R=20pt@C=40pt{
{\mathbb F^r} \ar[r]^D
\ar[d]_{\varphi}
&{\mathbb F^r}\ar[d]^{\varphi}
   \\
{\mathbb F^{r}} \ar[r]^{D'} &{\mathbb
F^{r}} }\] is commutative: $\varphi
D=D'\varphi$.

The following theorem is proved in
Section \ref{jdt} by the method for
constructing a regularizing
decomposition described in Section
\ref{reg}.

\begin{theorem}\label{yyw}
Each pair $(A,B)$ of matrices of the
same size over $\mathbb F=\mathbb R$ or
$\mathbb C$ is topologically equivalent
to a direct sum \eqref{1.4a}, in which
$D$ is an $r\times r$ nonsingular
matrix and each $(M_i,N_i)$ is of the
form \eqref{krx}. The matrix $D$ is
determined by $(A,B)$ uniquely up to
topological similarity; the summands
$(M_1,N_1),\dots, (M_t,N_t)$ are
determined uniquely up to permutation.
\end{theorem}

Kuiper and Robbin~\cite{Kuip-Robb,
Robb} gave a criterion for topological
similarity of real matrices without
eigenvalues that are roots of~$1$;
their result was extended to complex
matrices in \cite{bud1}. The problem of
topological similarity of matrices with
an eigenvalue that is a root of $1$ was
considered by Kuiper and
Robbin~\cite{Kuip-Robb, Robb}, Cappell
and Shaneson~\cite{Capp-conexamp,
Capp-2th-nas-n<=6,Capp-big-n<6,
Cap+sha,Cap+ste}, and Hambleton and
Pedersen \cite{h-p,h-p1}. The problem
of topological classification was
studied for orthogonal operators
\cite{Pardon}, for affine operators
\cite{Blanc,bud,bud1,Ephr}, for
M\"obius transformations
\cite{ryb+ser_meb}, for chains of
linear mappings \cite{ryb+ser}, for
oriented cycles of linear mappings
\cite{ryb_new,ryb+ser1}, and for quiver
representations \cite{lop}.


\section{A method for constructing a regularizing
decomposition}\label{reg}

\subsection{A formal description of the
method}

Let ${\mathcal P}:\pair{U}{V}{\mathcal
A}{\mathcal B}$ be a pair of linear
mappings over a field $\mathbb F$. Let
\begin{equation}\label{loz}
{\mathcal P}={\mathcal P}_{11},\ {\mathcal P}_{12},\ \dots,\
{\mathcal P}_{1\ell_1}={\mathcal P}_{21},\
{\mathcal P}_{22},\ \dots,\
{\mathcal P}_{2\ell_2}={\mathcal P}_{31},\
{\mathcal P}_{32},\ \dots,\
{\mathcal P}_{3\ell_3}
\end{equation}
be a sequence of pairs  of linear
mappings on vector spaces over $\mathbb
F$, in which
\begin{itemize}
  \item ${\mathcal P}_{i,j+1}$ is
      defined by ${\mathcal
      P}_{ij}:\pair{U_{ij}}{V_{ij}}{{\mathcal
      A}_{ij}}{{\mathcal B}_{ij}}$
      as follows:
\begin{align*}
&\xymatrix@R=20pt@C=20pt{
{{\mathcal P}_{1,j+1}}:&
{U_{1,j+1}:=\mathcal A_{1j}^{-1}
(\im \mathcal B_{1j})}
\ar@<0.3ex>[rr]^{\quad\
\mathcal A_{1,j+1}}
\ar@<-0.3ex>[rr]_{\quad\
\mathcal B_{1,j+1}}
&&{\im \mathcal B_{1j}=:V_{1,j+1}}
}\\
&\xymatrix@R=20pt@C=20pt{
{{\mathcal P}_{2,j+1}}:&
{U_{2,j+1}:=\mathcal B_{2j}^{-1}
(\im \mathcal A_{2j})}
\ar@<0.3ex>[rr]^{\quad\
\mathcal A_{2,j+1}}
\ar@<-0.3ex>[rr]_{\quad\
\mathcal B_{2,j+1}}
&&{\im \mathcal A_{2j}=:V_{2,j+1}}
}
       \\
&\xymatrix@R=20pt@C=20pt{
{{\mathcal P}_{3,j+1}}:&{U_{3,j+1}:=
U_{3j}/\Ker\mathcal B_{3j}}
\ar@<0.3ex>[rr]^{\mathcal A_{3,j+1}\quad}
\ar@<-0.3ex>[rr]_{\mathcal B_{3,j+1}\quad}
&&{V_{3j}/\mathcal A_{3j}
(\Ker\mathcal B_{3j})=:V_{3,j+1}}
}
\end{align*}
(the mappings $\mathcal A_{i,j+1}$
and $\mathcal B_{i,j+1}$ are
induced by $\mathcal A_{ij}$ and
$\mathcal B_{ij}$),

\item ${\mathcal
    P}_{i\ell_i}=(\mathcal
    A_{i\ell_i},\mathcal
    B_{i\ell_i})$ is the first pair
    in the sequence ${\mathcal
    P}_{i1},\ {\mathcal P}_{i2},\
    \dots$ such that
\begin{itemize}
  \item $\mathcal B_{i\ell_i}$ is
      surjective if $i=1$,
  \item $\mathcal A_{i\ell_i}$ is
      surjective if $i=2$,
  \item  $\mathcal B_{i\ell_i}$ is
      bijective if $i=3$.
\end{itemize}
\end{itemize}

\begin{theorem}\label{ykw}
Let ${\mathcal P}:\pair{U}{V}{\mathcal
A}{\mathcal B}$ be a pair of linear
mappings over a field $\mathbb F$, and
let \eqref{loz} be a sequence of pairs
of linear mappings on vector spaces
over $\mathbb F$ that has been
constructed according to the preceding
formal description. Then there exist
bases of the spaces $U$ and $V$ in
which $(\mathcal A,\mathcal B)$ is
given by the direct sum of the
following matrix pairs:
\begin{itemize}
 \item a pair of nonsingular
     matrices that gives ${\mathcal
    P}_{3\ell_3}=(\mathcal
    A_{3\ell_3},\mathcal
    B_{3\ell_3})$,

  \item $(\dim V_{1j}-\dim
      V_{1,j+1})-(\dim U_{1j} -\dim
      U_{1,j+1})$ copies of
      $(L_j^T,R_j^T)$, $1\le
      j<\ell_1$,

  \item $(\dim U_{1j}-\dim
      U_{1,j+1})-(\dim V_{1,j+1}
      -\dim V_{1,j+2})$ copies of
      $(I_j,J_j(0))$, $1\le
      j<\ell_1$,

  \item $\dim V_{2j}-2\dim
      V_{2,j+1}+\dim V_{2,j+2}$
      copies of $(J_j(0),I_j)$,
      $1\le j<\ell_2$,

  \item   $\dim U_{3j}-2\dim
      U_{3,j+1}+\dim
      U_{3,j+2}$\footnote{This sum
      was written in [Linear
      Algebra Appl. 450 (2014)
      121--137] incorrectly as
      follows: $\dim V_{3j}-2\dim
      V_{3,j+1}+\dim V_{3,j+2}$ .}
      copies of $(L_j,R_j)$, $1\le
      j<\ell_3$.
\end{itemize}
This direct sum is a regularizing
decomposition for each matrix pair that
gives ${\mathcal P}$ in some bases of
$U$ and $V$.
\end{theorem}

\subsection{A justification of the
method}

In this section we prove Theorem
\ref{ykw}. All matrices and vector
spaces are considered over a field
$\mathbb F$.

The \emph{direct sum} of pairs of
linear mappings ${\mathcal
P}:\pair{U}{V}{\mathcal A}{\mathcal B}$
and ${\mathcal
P'}:\pair{U'}{V'}{\mathcal A'}{\mathcal
B'}$ is the pair
\[
{\mathcal
P}\oplus{\mathcal
P'}:\ \ {\xymatrix@C=45pt{{U\oplus U'}
\ar@<0.3ex>[r]^{\mathcal A\oplus\mathcal A'}
\ar@<-0.3ex>[r]_{\mathcal
B\oplus\mathcal B'}&{V\oplus V'}}}
\]
It follows from Kronecker's canonical
form for matrix pencils \cite[Section
XII]{gan} (or from the Krull--Schmidt
theorem for quiver representations
\cite{ser_enc}) that each pair of
linear mappings ${\mathcal
P}:\pair{U}{V}{\mathcal A\:}{\mathcal
B\:}$ decomposes into a direct sum of
direct-sum-indecomposable pairs and
this sum is uniquely determined, up to
permutation and linear equivalence of
direct summands. Replacing all summands
that are pairs of bijections by their
direct sum ${\mathcal R}$, we obtain a
decomposition
\begin{equation}\label{qug}
\mathcal P=\mathcal R\oplus\mathcal S_1
\oplus\dots\oplus
\mathcal S_t
\end{equation}
in which $\mathcal R$ is a pair of
linear bijections and each $\mathcal
S_i$ is given in some bases of its
spaces by one of the matrix pairs
\eqref{krx}. Each row and each column
of the matrices in matrix pairs
\eqref{krx} contains at most one ``1''
and the other entries are ``0''; hence
the basis vectors of $\mathcal S_i$
form a chain
\begin{equation}\label{lkt}
{\mathcal C}_i:\ \bullet
\text{\,---}\bullet
\text{---}\cdots\text{---}\bullet
\text{---}\,\bullet
\end{equation}
in which the points denote these basis
vectors and each line is
$\xrightarrow{\mathcal A}$,
$\xrightarrow{\mathcal B}$,
$\xleftarrow{\mathcal A}$, or
$\xleftarrow{\mathcal B}$. The number
of lines is called the \emph{length} of
${\mathcal C}_i$. For each chain
$\mathcal C$ we denote by $P(\mathcal
C)$ the matrix pair that is determined
by $\cal C$. For example,
\begin{equation}\begin{split}\label{llv}
(I_3,J_3(0))&=
P(\stackrel{\mathcal
      B}{\nrightarrow}
      v_1\xleftarrow{\mathcal A}u_1
\xrightarrow{\mathcal B}
v_2\xleftarrow{\mathcal A}u_2
\xrightarrow{\mathcal B} v_3
\xleftarrow{\mathcal A}u_3
\xrightarrow{\mathcal B}0)
        \\
(L_3^T,R_3^T)&=
P(\stackrel{\mathcal B}{\nrightarrow}v_1
\xleftarrow{\mathcal A}u_1
\xrightarrow{\mathcal B}
v_2\xleftarrow{\mathcal A}u_2
\xrightarrow{\mathcal B} v_3
\stackrel{\mathcal
      A}{\nleftarrow})
\end{split}\end{equation}
in which $u_1,u_2,\dots$ are basis
vectors of $U$, $v_1,v_2,\dots$ are
basis vectors of $V$, and
$\stackrel{\mathcal B}{\nrightarrow}
v_1$ means that $v_1\notin \im \mathcal
B$.

Now we apply to the pair ${\mathcal
P}_1:={\mathcal P}$ transformations of
three types that transform the regular
part $\mathcal R$ to a linearly
equivalent pair and decrease the
lengths of chains \eqref{lkt}, which
correspond to the singular summands. We
describe how the singular summands are
changed. We repeat these
transformations until all singular
summands are eliminated.

\begin{rs}\label{lqw}
Replace a pair of linear mappings $
{\mathcal P}_1:
\pair{U_1}{V_1}{{\mathcal
A}_1}{{\mathcal B}_1} $ with the pair
\begin{equation}\label{lksp}
\xymatrix@R=20pt@C=40pt{
{\mathcal P_2:
\hspace{-1cm}}
&{U_2:=
\mathcal A_1^{-1}(\im \mathcal B_1)}
\ar@<0.3ex>[r]^{\quad\ \mathcal A_2}
\ar@<-0.3ex>[r]_{\quad\ \mathcal B_2}
&{\im \mathcal B_1=:V_2}}
\end{equation}
in which $\mathcal A_2$ and $\mathcal
B_2$ are the restrictions of $\mathcal
A_1$ and $\mathcal B_1$. We write
$\mathcal P_2=\alg_1({\mathcal P_1})$.
\end{rs}
The mappings ${\mathcal A}_2$ and
${\mathcal B}_2$ in \eqref{lksp} are
defined correctly since ${\mathcal
A}_1(U_2)\subset V_2$ and ${\mathcal
B}_1(U_2)\subset V_2$.

\begin{lemma}\label{knp}
Let $\mathcal P_1$ and $\mathcal P_2$
be the pairs of linear mappings from
Transformation 1. Let $(A_1,B_1)$ and
$({A}_2,{B}_2)$ be their matrix pairs
in arbitrary bases. We can construct a
regularizing decomposition of
$({A}_2,{B}_2)$ from a regularizing
decomposition of $(A_1,B_1)$ as
follows:
\begin{itemize}
  \item Delete all summands of the
      form
\begin{align}
\label{111}
&
      P(\stackrel{\mathcal
      B_1}{\nrightarrow}
      v_1\stackrel{\mathcal
A_1}{\nleftarrow})
      =(0_{10},0_{10})=(L_1^T,R_1^T)
                       \\
    \label{222}
&
      P(\stackrel{\mathcal
B_1}{\nrightarrow}
      v_1\xleftarrow{\mathcal
A_1}u_1 \xrightarrow{\mathcal
B_1}0)= ([1],[0])=(I_1,J_1(0));
\end{align}
their numbers are, respectively,
\begin{gather}\label{555}
(\dim V_1-\dim V_2)-(\dim U_1 -\dim
      U_2)
\\\label{666}
(\dim U_1-\dim U_2)-(\dim V_2 -\dim
      V_3),
\end{gather}
in which $V_3$ is the second vector
space of the pair
\begin{equation*}\label{lksy}
\xymatrix@R=20pt@C=40pt{
{\mathcal P_3:=\alg_1({\mathcal P_2}):
\hspace{-1cm}}
&{U_3:=
\mathcal A_2^{-1}(\im \mathcal B_2)}
\ar@<0.3ex>[r]^{\quad\ \mathcal A_3}
\ar@<-0.3ex>[r]_{\quad\ \mathcal B_3}
&{\im \mathcal B_2=:V_3}}.
\end{equation*}

  \item Replace all summands of the
      form
\begin{align*}
&P(\stackrel{\mathcal
B_1}{\nrightarrow}
      v_1\xleftarrow{\mathcal
A_1}u_1
      \xrightarrow{\mathcal
B_1}v_2\xleftarrow{\mathcal A_1}
      \cdots\xrightarrow{\mathcal
B_1}v_k
      \stackrel{\mathcal
A_1}{\nleftarrow})
      =(L_k^T,R_k^T)
      \\
&
      P(\stackrel{\mathcal
B_1}{\nrightarrow}
      v_1\xleftarrow{\mathcal
A_1}u_1
      \xrightarrow{\mathcal
B_1}v_2\xleftarrow{\mathcal A_1}
      \cdots\xrightarrow{\mathcal
B_1}v_k\xleftarrow{\mathcal A_1}
      u_k\xrightarrow{\mathcal
B_1}0)=(I_k,J_k(0))
\end{align*}
with $k\ge 2$ by
\begin{align*}
&
      P(\stackrel{\mathcal
B_2}{\nrightarrow}
      v_2\xleftarrow{\mathcal
A_2}u_2
      \xrightarrow{\mathcal
B_2}v_3\xleftarrow{\mathcal
A_2}
      \cdots\xrightarrow{\mathcal
B_2}
      v_k\stackrel{\mathcal
A_2}{\nleftarrow})
      =(L_{k-1}^T,R_{k-1}^T)
                       \\
&
      P(\stackrel{\mathcal
B_2}{\nrightarrow}
      v_2\xleftarrow{\mathcal
A_2}u_2
      \xrightarrow{\mathcal
B_2}v_3\xleftarrow{\mathcal
A_2}
      \cdots\xrightarrow{\mathcal
B_2}v_k
      \xleftarrow{\mathcal
A_2}
      u_k\xrightarrow{\mathcal
B_2}0)
      =(I_{k-1},J_{k-1}(0)),
\end{align*}
respectively.

  \item Leave the other summands
      unchanged.

\end{itemize}
\end{lemma}

\begin{proof}
Only in this proof, we use regularizing
decompositions with the pairs
$(J_k(0)^T,I_k)$ and $(R_k,L_k)$
instead of $(J_k(0),I_k)$ and
$(L_k,R_k)$ (they are equivalent). Each
chain \eqref{lkt} for such regularizing
decompositions has either the form
\begin{equation}\label{sel}
\stackrel{\mathcal
      B_1}{\nrightarrow}
      v_1\xleftarrow{\mathcal A_1}u_1
\xrightarrow{\mathcal B_1}
v_2\xleftarrow{\mathcal A_1}u_2
\xrightarrow{\mathcal B_1}\cdots,
\end{equation}
as in \eqref{llv}, or the form
\begin{equation}\label{kfy}
0
\xleftarrow{\mathcal A_1}
u_1 \xrightarrow{\mathcal B_1}
v_1\xleftarrow{\mathcal A_1}u_2
\xrightarrow{\mathcal B_1}\cdots,
\end{equation}
as in
\begin{align*}
(J_3(0)^T,I_3)&=
P(0
\xleftarrow{\mathcal A_1}
u_1 \xrightarrow{\mathcal B_1}
v_1\xleftarrow{\mathcal A_1}u_2
\xrightarrow{\mathcal B_1}
v_2\xleftarrow{\mathcal A_1}u_3
\xrightarrow{\mathcal B_1}v_3
\stackrel{\mathcal
      A_1}{\nleftarrow})
       \\
(R_3,L_3)&=P(0
\xleftarrow{\mathcal A_1}
u_1 \xrightarrow{\mathcal B_1}
v_1\xleftarrow{\mathcal A_1}u_2
\xrightarrow{\mathcal B_1}
v_2\xleftarrow{\mathcal A_1}u_3
\xrightarrow{\mathcal B_1}0).
\end{align*}
Applying Transformation 1 to
\eqref{qug} with $\mathcal P=\mathcal
P_1$, we get
\[
 \alg_1(\mathcal
P_1)=\alg_1(\mathcal R)\oplus
\alg_1(\mathcal S_1)
\oplus\dots\oplus
\alg_1(\mathcal S_t).
\]
It is clear that $\alg_1(\mathcal
R)=\mathcal R$.

Each $\mathcal S_i$ is given by a chain
$\mathcal C_i$ of the form \eqref{sel}
or \eqref{kfy}. Respectively,
$\alg_1(\mathcal S_i)$ is given by the
part of $\mathcal
C_i$ taken in a rectangle in \\[-5mm]
\begin{equation}\label{ewv}
\begin{split}
\text{\hbox{\raisebox{-2.6cm}
{$\mathcal C_i:$}}} \quad
\xymatrix@R=0pt@C=50pt{
  {\save
  .[ddddddddddr]!C*[F-:<3pt>]\frm{}
  \restore}
{\quad\ }&{\begin{matrix}{}\\\vdots
\end{matrix}}\\{\vdots}&\\&v_4\\\\
u_3\ar@{-->}[ruu]^{\mathcal B_1}
\ar[rdd]^{\mathcal
A_1}&\\\\
&v_3\\\\
u_2\ar@{-->}[ruu]^{\mathcal B_1}
\ar[rdd]^{\mathcal
A_1}&\\\\
&v_2\\\\
u_1\ar@{-->}[ruu]^{\mathcal B_1}
\ar[rdd]^{\mathcal
A_1}&\\\\
& v_1 }\quad\text{\hbox{\raisebox{-2.6cm}
{or}}}\quad
\xymatrix@R=0pt@C=50pt{
  {\save
  .[ddddddddddddr]!C*[F-:<3pt>]\frm{}
  \restore}
{\quad\
}&{\begin{matrix}{}\\\vdots\quad
\end{matrix}}\\{\vdots}&\\&v_3\quad\\\\
u_3\ar@{-->}[ruu]^{\mathcal
B_1}\ar[rdd]^{\mathcal
A_1}&\\\\
&v_2\quad\\\\
u_2\ar@{-->}[ruu]^{\mathcal
B_1}\ar[rdd]^{\mathcal
A_1}&\\\\
&v_1\quad\\\\
u_1\ar@{-->}[ruu]^{\mathcal B_1}&
{\begin{matrix}
\end{matrix}}}
\end{split}\end{equation}
since
\begin{itemize}
  \item in the pair of linear
      mappings given by the
      left-hand chain, the space
      $\im \mathcal B_1$ is
      generated by $v_2,v_3,\dots$
      and the space $\mathcal
      A_1^{-1}(\im \mathcal B_1)$
      is generated by
      $u_2,u_3,\dots$;

  \item in the pair of linear
      mappings given by the
      right-hand chain, the spaces
      $\im \mathcal B_1$ and
      $\mathcal A_1^{-1}(\im
      \mathcal B_1)$ are generated
      by $v_1,v_2,\dots$ and
      $u_1,u_2,\dots$ ($u_1\in
      \mathcal A_1^{-1}(\im
      \mathcal B_1)$ because
      $u_1\in\mathcal
      A_1^{-1}(0)$).
\end{itemize}

This proves all but two of the
statements in Lemma \ref{knp}; it
remains to prove the correctness of
\eqref{555} and \eqref{666}.

It follows from \eqref{ewv} that under
the action of Transformation 1 each
chain of the form \eqref{sel} loses its
basis vector $v_1$, each chain of the
form \eqref{sel} of nonzero length
loses its basis vector $u_1$, and each
chain of the form \eqref{kfy} does not
change. Hence, the number of chains of
the form \eqref{sel} is equal to $\dim
V_1-\dim V_2$. The number of chains of
the form \eqref{sel} of nonzero length
is equal to $\dim U_1-\dim U_2$. Thus,
the number of chains of the form
\begin{equation}\label{n1}
\stackrel{\mathcal
      B_1}{\nrightarrow}
      v_1\stackrel{\mathcal
A_1}{\nleftarrow}
\end{equation}
is equal to \eqref{555}.

The chains of the form \eqref{n1} and
\begin{equation}\label{n2}
\stackrel{\mathcal B_1}{\nrightarrow}
      v_1\xleftarrow{\mathcal
A_1}u_1 \xrightarrow{\mathcal B_1}0
\end{equation}
(whose length is 0 and 1) disappear
under the action of Transformation 1.
The other chains of the form
\eqref{sel} become the chains of the
form \eqref{sel} of the pair $\mathcal
P_2$ (but their length is reduced by
2); and so their number is $\dim
V_2-\dim V_3$. Thus, the number of
``short'' chains \eqref{n1} and
\eqref{n2} is
\begin{equation}\label{kre}
(\dim V_1-\dim
      V_2)-(\dim V_2 -\dim V_3).
\end{equation}
Since the number of
chains \eqref{n1} in $\mathcal P_1$ is
\eqref{555}, the number of chains
\eqref{n2} is \eqref{666}.
\end{proof}

We repeat Transformation 1 until we
obtain a pair $\mathcal P_{\ell}:
\pair{U_{\ell}}{V_{\ell}}{\mathcal
A_{\ell}}{\mathcal B_{\ell}} $, in
which ${\mathcal B}_{\ell}$ is
surjective. We denote by $ {\mathcal
P_1}: \pair{U_1}{V_1}{\mathcal
A_1}{\mathcal B_1}$ the pair obtained
and apply to it the following
transformation (in fact, we apply
Transformation 1 to
$\pair{U_1}{V_1}{\mathcal B_1}{\mathcal
A_1}$).

\begin{rs}\label{lqph}
If ${\mathcal P_1}:
\pair{U_1}{V_1}{\mathcal A_1}{\mathcal
B_1}$ is a  pair of linear mappings in
which ${\mathcal B_1}$ is surjective,
then replace it with the pair
\begin{equation}\label{lfr}
\xymatrix@R=20pt@C=40pt{
{\mathcal P_2=\alg_2({\mathcal P_1}):
\hspace{-1cm}}
&{U_2:=
\mathcal B_1^{-1}(\im \mathcal A_1)}
\ar@<0.3ex>[r]^{\quad\ \mathcal A_2}
\ar@<-0.3ex>[r]_{\quad\ \mathcal B_2}
&{\im \mathcal A_1=:V_2}}
\end{equation}
in which $\mathcal A_2$ and $\mathcal
B_2$ are the restrictions of $\mathcal
A_1$ and $\mathcal B_1$.
\end{rs}

\begin{lemma}\label{knp1}
Let $\mathcal P_1$ and $\mathcal P_2$
be the pairs of linear mappings from
Transformation 2. Let $(A_1,B_1)$ and
$({A}_2,{B}_2)$ be their matrix pairs
in arbitrary bases. We can construct a
regularizing decomposition of
$({A}_2,{B}_2)$ from a regularizing
decomposition  of $(A_1,B_1)$ as
follows:
\begin{itemize}
  \item Delete all summands of the
      form
\begin{equation*}
    \label{22}
      P(\stackrel{\mathcal
A_1}{\nrightarrow}
      v_1\xleftarrow{\mathcal
B_1}u_1 \xrightarrow{\mathcal
A_1}0)= ([0],[1])=(J_1(0),I_1);
\end{equation*}
their number is
\begin{equation*}\label{66}
\dim V_1-2\dim
      V_2+\dim V_3
\end{equation*}
in which $V_3$ is the second vector
space of the pair ${\mathcal
P}_3:=\alg_2({\mathcal P}_2)$.

  \item Replace all summands of the
      form
\begin{equation*}
      P(\stackrel{\mathcal
A_1}{\nrightarrow}
      v_1\xleftarrow{\mathcal
B_1}u_1
      \xrightarrow{\mathcal
A_1}v_2\xleftarrow{\mathcal B_1}
      \cdots\xrightarrow{\mathcal
A_1}v_k\xleftarrow{\mathcal B_1}
      u_k\xrightarrow{\mathcal
A_1}0)=(J_k(0),I_k)
\end{equation*}
with $k\ge 2$ by
\[
      P(\stackrel{\mathcal
A_2}{\nrightarrow}
      v_2\xleftarrow{\mathcal
B_2}u_2
      \xrightarrow{\mathcal
A_2}v_3\xleftarrow{\mathcal
B_2}
      \cdots\xrightarrow{\mathcal
A_2}v_k
      \xleftarrow{\mathcal
B_2}
      u_k\xrightarrow{\mathcal
A_2}0)
      =(J_{k-1}(0),I_{k-1}).
\]
\end{itemize}
\end{lemma}

\begin{proof} This lemma follows from
Lemma \ref{knp} applied to $\mathcal
P_1: \pair{U_1}{V_1}{\mathcal
B_1}{\mathcal A_1}$ in which $\mathcal
B_1$ is surjective, and so its singular
summands cannot be given by chains of
the form
\[
\stackrel{\mathcal A_1}{\nrightarrow}
      v_1\xleftarrow{\mathcal
B_1}u_1
      \xrightarrow{\mathcal
A_1}v_2\xleftarrow{\mathcal B_1}
      \cdots\xrightarrow{\mathcal
A_1}v_k
      \stackrel{\mathcal
B_1}{\nleftarrow},\qquad k\ge 1.
\]
Hence the number of chains of the form
$\stackrel{\mathcal A_1}{\nrightarrow}
      v_1\xleftarrow{\mathcal
B_1}u_1 \xrightarrow{\mathcal A_1}0$ is
equal to \eqref{kre}.
\end{proof}

We repeat Transformation 2 until we
obtain a pair $\mathcal P_{\ell}:
\pair{U_{\ell}}{V_{\ell}}{\mathcal
A_{\ell}}{\mathcal B_{\ell}} $, in
which ${\mathcal A}_{\ell}$ is
surjective (these transformations
preserve the surjectivity of ${\mathcal
B}_{\ell}$). We denote by $ {\mathcal
P_1}: \pair{U_1}{V_1}{\mathcal
A_1}{\mathcal B_1}$ the pair obtained
and apply to it the following
transformation.

\begin{rs}\label{lke}
If $ {\mathcal P_1}:
\pair{U_1}{V_1}{\mathcal A_1}{\mathcal
B_1}$ is a  pair of linear mappings in
which ${\mathcal A_1}$ and ${\mathcal
B_1}$ are surjective, then replace it
with the pair
\begin{equation}\label{muj}
\xymatrix@R=20pt@C=40pt{
{\mathcal P_2=\alg_3({\mathcal P_1}):
\hspace{-1cm}}
&{U_2:=U_1/\Ker\mathcal B_1}
\ar@<0.3ex>[r]^{\mathcal A_2\ \ }
\ar@<-0.3ex>[r]_{\mathcal B_2\ \ }
&{V_1/\mathcal A_1(\Ker\mathcal B_1)
=:V_2}}.
\end{equation}
\end{rs}

The mappings ${\mathcal A}_2$ and
${\mathcal B}_2$ in \eqref{muj} are
defined correctly since for each $u\in
U_1$ we have $\mathcal
A_1(u+\Ker\mathcal B_1)= \mathcal
A_1(u)+\mathcal A_1(\Ker\mathcal B_1)$
and $\mathcal B_1(u+\Ker\mathcal B_1)=
\mathcal B_1(u)$.

\begin{lemma}\label{kjr}
Let $\mathcal P_1$ and $\mathcal P_2$
be the pairs of linear mappings from
Transformation 3. Let $(A_1,B_1)$ and
$({A}_2,{B}_2)$ be their matrix pairs
in arbitrary bases. We can construct a
regularizing decomposition of
$({A}_2,{B}_2)$ from a regularizing
decomposition of $(A_1,B_1)$ as
follows:
\begin{itemize}
  \item Delete all summands of the
      form
      \[P(0\xleftarrow{\mathcal
      B_1}u_1 \xrightarrow{\mathcal
      A_1}0)
      =(0_{01},0_{01})=(L_1,R_1);\]
their number is
\begin{equation}\label{jst}
\dim U_1-2\dim
      U_2+\dim U_3
\end{equation}
in which $U_3$ is the first vector
space of the pair ${\mathcal
P}_3:=\alg_3({\mathcal P}_2)$.

  \item Replace all summands of the
      form
\begin{equation*}\label{jyz}
P( 0\xleftarrow{\mathcal B_1}
u_1\xrightarrow{\mathcal A_1}v_1
\xleftarrow{\mathcal B_1}u_2
\xrightarrow{\mathcal A_1}\cdots
\xleftarrow{\mathcal B_1}u_{k}
\xrightarrow{\mathcal A_1}0)
=(L_k,R_k),\quad k\ge 2,
\end{equation*}
by
\[
P(\bar 0\xleftarrow
{\mathcal B_2}\bar u_2
\xrightarrow{\mathcal A_2}\bar v_2
\xleftarrow{\mathcal B_2}\bar u_3
\xrightarrow{\mathcal A_2}\cdots
\xleftarrow{\mathcal B_2}
\bar u_{k}\xrightarrow{\mathcal A_2}\bar 0)
=(L_{k-1},R_{k-1}),
\]
in which $\bar 0, \bar u_2, \bar
v_2,\dots$ are the elements of the
factor spaces that contain $0,
u_2,v_2,\dots$\,.
\end{itemize}
\end{lemma}

\begin{proof}
Applying Transformation 3 to
\eqref{qug} with $\mathcal P=\mathcal
P_1$, we get
\[
 \alg_3(\mathcal
P_1)=\alg_3(\mathcal R)\oplus
\alg_3(\mathcal S_1)
\oplus\dots\oplus
\alg_3(\mathcal S_t).
\]
The pairs $\alg_3(\mathcal R)$ and
$\mathcal R$ are linearly equivalent.

Since $\mathcal A_1$ and $\mathcal B_1$
in $\mathcal P_1$ are surjections, the
end vectors of each chain of basis
vectors belong to $U_1$. Hence, the
chain of each $\mathcal S_i$ has the
form
\begin{equation}\label{net}
\mathcal C_i:\qquad
0\xleftarrow{\mathcal B_1}u_1
\xrightarrow{\mathcal A_1}v_1
\xleftarrow{\mathcal B_1}u_2
\xrightarrow{\mathcal A_1}\cdots
\xleftarrow{\mathcal B_1}u_{k}
\xrightarrow{\mathcal A_1}0,\qquad k\ge 1.
\end{equation}
Since $u_1\in \Ker\mathcal B_1$ and
$v_1\in\mathcal A_1(\Ker\mathcal B_1)$,
$\alg_3(\mathcal S_i)$ is given by the
chain
\begin{equation}\label{brt}
\bar{\mathcal C_i}:\qquad \bar 0\xleftarrow
{\mathcal B_2}\bar u_2
\xrightarrow{\mathcal A_2}\bar v_2
\xleftarrow{\mathcal B_2}\bar u_3
\xrightarrow{\mathcal A_2}\cdots
\xleftarrow{\mathcal B_2}
\bar u_{k}\xrightarrow{\mathcal A_2}\bar 0.
\end{equation}
Hence, the number of chains $\mathcal
C_i$ is equal to $\dim U_1-\dim U_2$.

The chains of the form
\begin{equation}\label{n2sd}
0\xleftarrow{\mathcal
      B_1}u_1 \xrightarrow{\mathcal
      A_1}0
\end{equation}
disappear under the action of
Transformation 3. The other chains
\eqref{net} become the chains
\eqref{brt} of the pair $\mathcal P_2$
(but their length is reduced by 2);
their number is $\dim U_2-\dim U_3$.
Thus, the number of chains \eqref{n2sd}
is \eqref{jst}.
\end{proof}

We repeat Transformation 3 until we
obtain a pair $\mathcal P_{\ell}:
\pair{U_{\ell}}{V_{\ell}}{\mathcal
A_{\ell}}{\mathcal B_{\ell}} $ in which
${\mathcal B}_{\ell}$ is bijective.
Since $\dim U_{\ell}=\dim V_{\ell}$ and
${\mathcal A}_{\ell}$ is surjective, it
is bijective too. Thus, $\mathcal
P_{\ell}$ is a regular pair.

\begin{proof}[Proof of Theorem
\ref{ykw}] We take an arbitrary pair of
linear mappings $\mathcal P:
\pair{U}{V}{\mathcal A}{\mathcal B}$
and an arbitrary decomposition
\eqref{qug} in which $\mathcal R$ is a
pair of linear bijections and each
$\mathcal S_i$ is a singular
indecomposable summand. We apply
Transformations \ref{lqw}--\ref{lke} to
$\mathcal P$ and obtain a pair that is
linearly equivalent to $\mathcal R$ in
\eqref{qug}. Lemmas
\ref{knp}--\ref{kjr} determine the
summands $\mathcal S_1,\dots,\mathcal
S_t$ uniquely up to linear equivalence.
\end{proof}


\section{A topological
classification of matrix pencils}
\label{jdt}

In this section we prove Theorem
\ref{yyw}. All matrices and vector
spaces are considered over  $\mathbb
F=\mathbb R$ or $\mathbb C$.

By Kronecker's theorem, each matrix
pair is equivalent (and hence
topologically equivalent) to
\eqref{1.4a}; it remains to prove the
uniqueness of \eqref{1.4a}. More
precisely, let
\begin{align}\label{kux}
(A,B)&:=(I_{r},D)\oplus(M_1,N_1)\oplus
\dots\oplus
(M_{t},N_{t})
      \\ \label{lis}
(A',B')&:=(I_{r'},D')\oplus(M'_1,N'_1)\oplus
\dots\oplus (M'_{t'},N'_{t'})
\end{align}
be direct sums of the form
\eqref{1.4a}; i.e., $D$ and $D'$ are
nonsingular and all $(M_i,N_i),$
$(M'_j,N'_j)$ are pairs of the form
\eqref{krx}. Let $(A,B)$ and $(A',B')$
be topologically equivalent. We need to
prove that
\begin{equation}\label{kjt}
\parbox[c]{0.7\textwidth}{$D$ is topologically similar to $D'$,
$t=t'$, and there is a reindexing of
the $(M'_j,N'_j)$'s such that
$(M_i,N_i)=(M'_i,N'_i)$ for all $i$.
}
\end{equation}

Let $ {\mathcal
P_1}:\pair{U_1}{V_1}{\mathcal
A_1}{\mathcal B_1} $ and $ {\mathcal
P_1}':\pair{U'_1}{V_1'}{\mathcal A_1'}
{\mathcal B_1'}$ be pairs of linear
mappings given by $(A,B)$ and $(A',B')$
in some bases of inner product spaces
$U_1,V_1,U_1',V_1'$ (which are
Euclidean if $\F=\R$ or unitary if
$\F=\C$). The commutative diagram
\eqref{kjlj} takes the form
\begin{equation}\label{kjlj1}
\begin{split}
\xymatrix@R=20pt@C=40pt{
{\mathcal P_1:\hspace{-7mm}}&
{U_1} \ar@<0.3ex>[r]^{\mathcal A_1}
\ar@<-0.3ex>[r]_{\mathcal B_1}
\ar[d]_{\varphi_1}
&{V_1}\ar[d]^{\psi_1}
   \\
{\mathcal P'_1:\hspace{-7mm}}&
{U'_1} \ar@<0.3ex>[r]^{\mathcal A_1'}
\ar@<-0.3ex>[r]_{\mathcal B_1'}&{V_1'}
}
\end{split}
\end{equation}
in which $\varphi_1$ and $\psi_1$ are
homeomorphisms.

Let us apply Transformations 1--3 to
the diagram \eqref{kjlj1}.

\begin{Rs}\label{lqphd}
Replace the commutative diagram
\eqref{kjlj1} with
\begin{equation}\label{lj1}
\begin{split}
\xymatrix@R=20pt@C=40pt{
{\mathcal P_2:\hspace{-7mm}}&
{U_2=\mathcal A_1^{-1}(\im \mathcal B_1)}
\ar@<0.3ex>[r]^{\quad\ \mathcal A_2}
\ar@<-0.3ex>[r]_{\quad\ \mathcal B_2}
\ar@<3ex>[d]_{\varphi_2}
&{\im \mathcal B_1=V_2}
\ar@<-3ex>[d]^{\psi_2}
   \\
{\mathcal P'_2:\hspace{-7mm}}&
{U_2'=\mathcal A_1'^{-1}(\im
\mathcal B_1')}
\ar@<0.3ex>[r]^{\quad\ \mathcal A'_2}
\ar@<-0.3ex>[r]_{\quad\ \mathcal B'_2}&
{\im \mathcal B'_1=V_2'}
}
\end{split}
\end{equation}
in which $\varphi_2$ and $\psi_2$ are
the restrictions of $\varphi_1$ and
$\psi_1$.
\end{Rs}

\begin{lemma}\label{kus}
The homeomorphisms $\varphi_2$ and
$\psi_2$ in \eqref{lj1} are defined
correctly.
\end{lemma}

\begin{proof}
We need to show that $\varphi_1(U_2)=
U'_2$ and $\psi_1(V_2)= V'_2$. It
suffices prove
\begin{equation}\label{jkt}
\varphi_1(U_2)\subset U'_2,\qquad
\psi_1(V_2)\subset  V'_2
\end{equation}
since then we can take
\[
\xymatrix@R=20pt@C=40pt{
{U'_1}
\ar@<0.3ex>[r]^{\mathcal A'_1}
\ar@<-0.3ex>[r]_{\mathcal B'_1}
\ar[d]_{\varphi_1^{-1}}
&{V'_1}
\ar[d]^{\psi_1^{-1}}
   \\
{U_1}
\ar@<0.3ex>[r]^{\mathcal A_1}
\ar@<-0.3ex>[r]_{\mathcal B_1}
&
{V_1}
}
\]
instead of \eqref{kjlj1} and obtain
$\varphi_1^{-1}(U'_2)\subset U_2$ and
$\psi_1^{-1}(V'_2)\subset V_2$ instead
of \eqref{jkt}, which implies
$U'_2\subset \varphi_1(U_2)$ and
$V'_2\subset \psi_1 (V_2)$.

Let us prove the second inclusion in
\eqref{jkt}. Take $v\in V_2=\im
\mathcal B_1$. There exists $u\in U_1$
such that $\mathcal B_1(u)=v$. Since
the diagram
\[\xymatrix@R=20pt@C=40pt{
{u}\ar[r]^{\mathcal B_1}
\ar[d]_{\varphi_1}
&v
\ar[d]^{\psi_1}
   \\
{u'}\ar[r]^{\mathcal B'_1} &v' }
\]
is commutative, $\psi_1(v)=v'=\mathcal
B_1'(u')\in \im\mathcal B_1'$.

Let us prove the first inclusion in
\eqref{jkt}: $\varphi_1(\mathcal
A_1^{-1}(V_2))\subset \mathcal
A_1'^{-1}(V'_2)$. Take $u\in \mathcal
A_1^{-1}(V_2)$, then $v:=\mathcal
A_1(u)\in V_2$. By the second inclusion
in \eqref{jkt}, $v':=\psi_1(v)\in
V'_2$. Since the diagram
\[
\xymatrix@R=20pt@C=40pt{
{u}\ar[r]^{\mathcal A_1}
\ar[d]_{\varphi_1}
&v
\ar[d]^{\psi_1}
   \\
{\varphi_1(u)}\ar[r]^{\ \ \mathcal A_1'} &v' }
\]
is commutative, $\varphi_1(u)\in
\mathcal A_1'^{-1}(V_2')$, which
completes the proof of correctness of
the mappings in \eqref{lj1}.
\end{proof}

Thus, the spaces $U_1,$ $V_1,$ $U_2,$
$V_2$ are homeomorphic to $U_1',$
$V_1',$ $U_2',$ $V_2'$ and so their
dimensions are equal. By Lemma
\ref{knp}, all regularizing
decompositions of $(A,B)$ and $(A',B')$
have the same number of summands
$(L_1^T,R_1^T)$ and the same number of
summands $(I_1,J_1(0))$.

We repeat Transformation \ref{lqphd}
until we obtain a diagram
\begin{equation}\label{ktu}
\begin{split}
\xymatrix@R=20pt@C=40pt{
{\mathcal P_{\ell}:\hspace{-7mm}}&
{U_{\ell}}
\ar@<0.3ex>[r]^{\mathcal A_{\ell}}
\ar@<-0.3ex>[r]_{\mathcal B_{\ell}}
\ar[d]_{\varphi_{\ell}}
&{V_{\ell}}
\ar[d]^{\psi_{\ell}}
   \\
{\mathcal P'_{\ell}:\hspace{-7mm}}&
{U_{\ell}'}
\ar@<0.3ex>[r]^{\mathcal A'_{\ell}}
\ar@<-0.3ex>[r]_{\mathcal B'_{\ell}}
&
{V_{\ell}'}
}
\end{split}
\end{equation}
in which ${\mathcal B}_{\ell}$ is a
surjection (then ${\mathcal B}'_{\ell}$
is a surjection too). An
$(\ell-1)$-fold application of Lemma
\ref{knp} ensures that
\begin{equation}\label{hri}
\parbox[c]{0.7\textwidth}{all
regularizing decompositions of
$(A,B)$ and $(A',B')$ have the same
number of summands $(L_k^T,R_k^T)$ and
the same number of summands
$(I_k,J_k(0))$ for each $k=1,2,\dots$
}
\end{equation}

Let \eqref{kjlj1} be the diagram
\eqref{ktu}  obtained. We apply to it
the following transformation (in fact,
we apply Transformation \ref{lqphd} to
the diagram obtained from \eqref{kjlj1}
by interchanging $\mathcal A_1$ and
$\mathcal B_1$ in $\mathcal P_1$ and
${\mathcal A_1}'$ and ${\mathcal B_1}'$
in ${\mathcal P_1}'$).


\begin{Rs}\label{lle}
If ${\mathcal B_1}$ and ${\mathcal
B_1}'$ in \eqref{kjlj1} are
surjections, then replace the
commutative diagram \eqref{kjlj1} with
\begin{equation*}\label{lkt1}
\begin{split}
\xymatrix@R=20pt@C=40pt{
{\mathcal P_2:\hspace{-7mm}}&
{U_2=\mathcal B_1^{-1}(\im \mathcal A_1)}
\ar@<0.3ex>[r]^{\quad\ \mathcal A_2}
\ar@<-0.3ex>[r]_{\quad\ \mathcal B_2}
\ar@<3ex>[d]_{\varphi_2}
&{\im \mathcal A_1=V_2}
\ar@<-3ex>[d]^{\psi_2}
   \\
{\mathcal P'_2:\hspace{-7mm}}&{U_2'
=\mathcal B_1\,'^{-1}(\im \mathcal A_1')}
\ar@<0.3ex>[r]^{\quad\ \mathcal A'_2}
\ar@<-0.3ex>[r]_{\quad\ \mathcal B'_2}&
{\im \mathcal A_1'=V_2'}
}
\end{split}
\end{equation*}
in which $\varphi_2$ and $\psi_2$ are
the restrictions of $\varphi_1$ and
$\psi_1$.
\end{Rs}

We repeat Transformation \ref{lle}
until we obtain a diagram \eqref{ktu}
in which ${\mathcal A}_{\ell}$ is a
surjection (then ${\mathcal A}'_{\ell}$
is a surjection too). By \eqref{hri} in
which the matrices of the pairs are
interchanged, all regularizing
decompositions of $(A,B)$ and $(A',B')$
have the same number of summands
$(J_k(0),I_k)$ for each $k=1,2,\dots$

Let \eqref{kjlj1} be the diagram
obtained. We apply to it the following
transformation.

\begin{Rs}\label{lke1}
If ${\mathcal A_1}$, ${\mathcal A_1}'$,
${\mathcal B_1}$, ${\mathcal B_1}'$ in
\eqref{kjlj1} are surjections, then
replace the commutative diagram
\eqref{kjlj1} with
\begin{equation}\label{lyx}
\begin{split}
\xymatrix@R=20pt@C=40pt{
{\mathcal P_2:\hspace{-7mm}}&
{U_2:=U_1/\Ker\mathcal B_1}
\ar@<0.3ex>[r]^{\mathcal A_2\ \ }
\ar@<-0.3ex>[r]_{\mathcal B_2\ \ }
\ar@<3ex>[d]_{\varphi_2}
&{V_1/\mathcal A_1(\Ker\mathcal B_1)=:V_2}
\ar@<-3ex>[d]^{\psi_2}
   \\
{\mathcal P'_2:\hspace{-7mm}}&
{U_2':=U_1'/\Ker\mathcal B_1'}
\ar@<0.3ex>[r]^{\mathcal A'_2\ \ }
\ar@<-0.3ex>[r]_{\mathcal B'_2\ \ }&
{V_1'/\mathcal A_1'(\Ker\mathcal B_1')=:V_2'}
}
\end{split}
\end{equation}
in which $\varphi_2$ and $\psi_2$ are
induced by $\varphi_1$ and $\psi_1$.
\end{Rs}

\begin{lemma}\label{kld}
The homeomorphisms $\varphi_2$ and
$\psi_2$ in \eqref{lyx} are defined
correctly.
\end{lemma}

\begin{proof}
We need to show that
\begin{align*}
\varphi_1(u+\Ker \mathcal B_1)=&
\varphi_1(u)+\Ker \mathcal B_1'\\
\psi_1(v+\mathcal A_1(\Ker \mathcal B_1))=&
\psi_1(v)+\mathcal A_1'(\Ker \mathcal B_1')
\end{align*}
for each $u\in U_1$ and $v\in V_1$.
Since $\mathcal A_1$ and $\mathcal B_1$
are surjections, it suffices to prove
\begin{align}\label{kue}
\varphi_1(u+\Ker \mathcal B_1)&\subset
\varphi_1(u)+\Ker \mathcal B_1'\\
\label{kue1}
\psi_1(v+\mathcal A_1(\Ker
\mathcal B_1))&\subset
\psi_1(v)+\mathcal A_1'
(\Ker \mathcal B_1').
\end{align}

Let us prove \eqref{kue}. Take $u\in
U_1$. Write $v:=\mathcal B_1(u)$ and
$v':=\psi_1(v)$. For each $k\in\Ker
\mathcal B_1$, we have
\[
\xymatrix@R=20pt@C=40pt{
{u+k}\ar[r]^{\quad\mathcal B_1}
\ar[d]_{\varphi_1}
&v
\ar[d]^{\psi_1}
   \\
{\varphi_1(u+k)}\ar[r]^{\qquad
\mathcal B_1'}
 &v' }
\]
hence $\varphi_1(u+k)\in\mathcal
B_1'^{-1}(v')=\varphi_1(u)+\Ker
\mathcal B_1'$.

Let us prove \eqref{kue1}. Take any
$v\in V_1$ and $k\in\Ker \mathcal B_1$.
Since $\mathcal A_1$ is surjective,
$v=\mathcal A_1(u)$ for some $u\in
U_1$, which gives
\[
\xymatrix@R=20pt@C=40pt{
{u}\ar[r]^{\mathcal A_1}
\ar[d]_{\varphi_1}
&v
\ar[d]^{\psi_1}
   \\
{\varphi_1(u)}\ar[r]^{\mathcal A_1'\quad}
&{\mathcal A_1'(\varphi_1(u))=\psi_1 (v)
\hspace{-1.5cm}}}
\]
By \eqref{kue}, there exists $k'\in\Ker
\mathcal B_1'$ such that
$\varphi_1(u+k)=\varphi_1(u)+k'$. Then
\[
\xymatrix@R=20pt@C=40pt{
{u+k}\ar[r]^{\mathcal A_1}
\ar[d]_{\varphi_1}
&{v+\mathcal A_1(k)}
\ar[d]^{\psi_1}
   \\
{\varphi_1(u)+k'}\ar[r]^{\mathcal A_1'\ \ }
&{\psi_1(v)+\mathcal A_1'(k')}
}
\]
which proves \eqref{kue1} and completes
the proof of correctness of the
mappings in \eqref{lyx}.
\end{proof}

Thus, the spaces $V_1$, $V_2$, and,
analogously, $V_3$ are homeomorphic to
$V_1'$, $V_2'$, and $V_3'$ and so their
dimensions are equal. By Lemma
\ref{kjr}, all regularizing
decompositions of $(A,B)$ and $(A',B')$
have the same number of summands
$(L_1,R_1)$.

We repeat Transformation \ref{lke1}
until we obtain a diagram \eqref{ktu}
in which ${\mathcal B}_{\ell}$ is a
bijection; then ${\mathcal A}_{\ell}$,
${\mathcal B}_{\ell}$, and ${\mathcal
B}'_{\ell}$ are bijections too. Thus,
$\mathcal P_{\ell}$ and $\mathcal
P'_{\ell}$ are topologically
equivalent; they are regular parts of
$\mathcal P_1$ and $\mathcal P_1'$. An
$(\ell-1)$-fold application of Lemma
\ref{kjr} ensures that all regularizing
decompositions of $(A,B)$ and $(A',B')$
have the same number of summands
$(L_k,R_k)$ for each $k=1,2,\dots$

\begin{proof}[Proof of Theorem
\ref{yyw}] We take arbitrary
regularizing decompositions \eqref{kux}
and \eqref{lis} that are topologically
equivalent, construct the commutative
diagram \eqref{kjlj1}, apply
Transformations \ref{lqphd}--\ref{lke1}
 to it, and obtain that
\begin{itemize}
  \item $t=t'$,
  \item there is a reindexing of
      the $(M'_j,N'_j)$'s such that
      $(M_i,N_i)=(M'_i,N'_i)$ for
      all $i$, and
  \item $(I_r,D)$ is topologically
      equivalent to $(I_{r'},D')$.
\end{itemize}
Thus, there are homeomorphisms
${\varphi}$ and $\psi$ such that the
diagram
\[
\xymatrix@R=20pt@C=40pt{
{\mathbb F^r} \ar@<0.3ex>[r]^{I_r}
\ar@<-0.3ex>[r]_D
\ar[d]_{\varphi}
&{\mathbb F^r}\ar[d]^{\psi}
   \\
{\mathbb F^{r'}}
\ar@<0.3ex>[r]^{I_{r'}}
\ar@<-0.3ex>[r]_{D'}&{\mathbb F^{r'}}
}\] is commutative. Hence, $r=r'$,
${\varphi}=\psi$, and so $D$ is
topologically similar to $D'$, which
ensures \eqref{kjt} and completes the
proof of Theorem \ref{yyw}.
\end{proof}

\section*{Acknowledgement}

V. Futorny  is supported in part by the
CNPq (grant 301320/2013-6) and FAPESP
(grant 2010/50347-9). This work was
done during a visit of V.V. Sergeichuk
to the University of S\~ao Paulo. He is
grateful to the University of S\~ao
Paulo for hospitality and FAPESP for
financial support (grant 2012/18139-2).

\end{document}